\documentclass[a4paper,reqno,11pt]{amsart}

\usepackage{amsmath, amssymb, amsthm, enumerate,amsfonts,latexsym, mathrsfs}
\usepackage{stmaryrd}

\usepackage[top=30mm,right=30mm,bottom=30mm,left=30mm]{geometry}

 \newtheorem{theorem}{Theorem}[section]
\newtheorem{corollary}[theorem]{Corollary}
\newtheorem{lemma}[theorem]{Lemma}

\newtheorem*{Hconjecture}{Huppert's Conjecture}

\theoremstyle{definition}
\newtheorem{remark}[theorem]{Remarks}

\newcommand{\PSL}{{\mathrm {PSL}}}

\newcommand{\GL}{{\mathrm {GL}}}

\newcommand{\la}{\langle}
\newcommand{\ra}{\rangle}
\newcommand{\Z}{\mathbb{Z}}

\newcommand{\F}{\mathbb{F}}
\newcommand{\N}{\mathbb{N}}
\newcommand{\Alt}{\mathrm{A}}
\newcommand{\Sym}{\mathrm{S}}
\newcommand{\Irr}{{\mathrm {Irr}}}
\newcommand{\cd}{{\mathrm {cd}}}
\newcommand{\St}{{\mathrm {St}}}

\newcommand{\Kernel}{{\mathrm {Ker}}}
\newcommand{\Centralizer}{{\mathbf {C}}}
\newcommand{\Center}{{\mathbf {Z}}}
\providecommand{\Aut}{\mathop{\rm Aut}\nolimits}%
\providecommand{\Out}{\mathop{\rm Out}\nolimits}%

\begin{document}

\title[Huppert's Conjecture]{Huppert's Conjecture for alternating groups}

\author[C.\ Bessenrodt]{Christine Bessenrodt}

\address{Institut f\"ur Algebra, Zahlentheorie und Diskrete Mathematik,
Leibniz Universit\"at Hannover, Welfengarten 1,
D-30167 Hannover, Germany}
\email{bessen@math.uni-hannover.de}

\author[H.P.\ Tong-Viet]{Hung P. Tong-Viet${}^\dag$}

\address{Department of Mathematics and Applied Mathematics,
University of Pretoria,
Private Bag X20, Hatfield
0028 Pretoria,
South Africa}
\email{Hung.Tong-Viet@up.ac.za}
\thanks{$^{\dag}$ This work is based on the research supported in part by the National Research Foundation of South Africa (Grant Number 93408)}

\author[J.\ Zhang]{Jiping Zhang${}^*$}

\thanks{$^{*}$Supported  by NSFC(11131001) and National
973 project(2011CB808003). }
\address{Beijing International Center for
Mathematical Research\\Lmam, The School of Mathematical
Sciences, \\ Peking University, Beijing, P. R. China
}
\email{jzhang@pku.edu.cn}
\keywords{character degrees,
alternating groups, Huppert's Conjecture} \subjclass[2000]{Primary
20C15, 20D05, 20C30}
\date{January 20, 2014}

\begin{abstract}  We prove that the alternating groups of degree at least $5$ are uniquely determined up to an abelian direct factor by the degrees of their irreducible complex representations.
This confirms Huppert's Conjecture for alternating groups.
\end{abstract}

\maketitle

\section{Introduction}\label{sec:intro}
Let $G$ be a finite group.
Denote by $\Irr(G)$ the set of all complex irreducible characters of~$G$.
Let $\cd(G)$ be the set of all
irreducible character degrees of $G$ forgetting multiplicities, that is,
$$\cd(G)=\{\chi(1)\mid \chi\in \Irr(G)\}.$$
In \cite{Hupp}, Huppert proposed the following
conjecture.

\begin{Hconjecture} Let $H$
be any finite nonabelian simple group and $G$ be a finite group
such that $\cd(G)=\cd(H)$.
Then $G\cong H\times A$, where $A$ is abelian.
\end{Hconjecture}

Notice that Huppert's Conjecture is best possible in the sense that if $G=H\times A$ with $A$ abelian, then $\cd(G)=\cd(H)$. In this paper, we prove the following result.

\begin{theorem}\label{th:main} Let $5\le n\in \N$. Let $G$ be a finite group such that $\cd(G)=\cd(\Alt_n)$. Then $G\cong \Alt_n\times A$, where $A$ is abelian.
\end{theorem}

This verifies Huppert's Conjecture for all alternating groups and is a major step toward the proof of the conjecture. This also extends the main result obtained by the second author in \cite{Tong1}. We now describe our approach to the proof of Huppert's Conjecture for alternating groups.
Suppose that $G$ is a finite group and $H$ is a finite nonabelian simple group such that $\cd(G)=\cd(H)$. To verify Huppert's Conjecture for the simple group $H$, we need to prove the following.

\smallskip
{\bf Step~1:} Show that $G$ is nonsolvable;

\smallskip
{\bf Step~2:}  If $L/M$ is any nonabelian chief factor of $G$,  then $L/M\cong H;$

\smallskip
{\bf Step~3:}  If $L$ is a finite perfect group and $M$ is a minimal normal elementary abelian subgroup of $L$ such that $L/M\cong H$,  then some degree of $L$ divides no degree of $H$.

\smallskip
{\bf Step~4:}  If $T$ is any finite group with $H\unlhd T\leq \Aut(H)$ and $T\neq H$,  then $\cd(T)\nsubseteq \cd(H)$.

\begin{remark} We want to make a few remarks.

(1) The method given here is basically Huppert's strategy as described in \cite{Hupp}. However, we add some improvements. We also combined Steps $3$ and $4$ of Huppert's method into one step (Step $3$). Notice that these steps are interchangeable.

(2) In the proof of Theorem \ref{th:main}, we can assume that $n\ge 14$. Huppert proved the conjecture in many cases, including
alternating groups of degree up to $n=11$;
for $n=12$ and $13$, it was proved by H.N.\ Nguyen,
H.P.\ Tong-Viet and T.P.\ Wakefield in \cite{NTW}.

(3) In verifying Step~1 it is essential that $H$ is simple.  Indeed, G.\ Navarro \cite{Navarro} recently constructed a finite perfect group $H$ and a finite solvable group $G$ such that $\cd(G)=\cd(H)$. More surprisingly, Navarro and Rizo \cite{NR} found a finite perfect group $H$ and a finite nilpotent group $G$ with $\cd(G)=\cd(H)$.
It remains open whether the character degrees together with their multiplicities can determine the solvability of a finite group.
This is related to Brauer's Problem 2 \cite{Brauer}, which asks
when nonisomorphic groups have isomorphic group algebras. 
Our proof of Step 1 uses
a result of G.R.\ Robinson \cite{Robinson} on the minimal degree of  nonlinear irreducible characters of finite solvable groups.

(4) To verify Step $2$, we will use the classification of finite simple groups in conjunction with the classification of prime power degree representations of alternating groups, symmetric groups and their covers \cite{BBOO,Bessenrodt} and the small degree representations of alternating groups \cite{Rasala}.

(5) In proving Step~3 for $H=\Alt_n$, we use a result due to Guralnick and Tiep \cite{Guralnick} on the non-coprime $k(GV)$ problem. Unfortunately, this only works for $n\ge 17$. For the remaining values of $n$, we have to resort
to Huppert's original method (see Theorem \ref{th:largedegree}). Up to this point, we have been able to show that either $G\cong \Alt_n\times A$ or $G\cong (\Alt_n\times A)\cdot 2$ and $G/A\cong\Sym_n$ with $A$ abelian (see Theorem \ref{thm:almostHC}).

(6) Finally, Theorem \ref{th:main} follows  if one can show that $\cd(\Sym_n)\nsubseteq \cd(\Alt_n)$ which is Step 4. (Recall that we assume $n\ge 14$.) Indeed, it is conjectured in \cite{Tong2} that if $\lambda=(k+1,1^k)$ when $n=2k+1$ and $\lambda=(k,2,1^{k-2})$ when $n=2k$, then
$\chi^{\lambda}(1)\in {\cd}(\Sym_n)\setminus{\cd}( \Alt_n)$,
and  $\chi^{\lambda}(1)/2\in {\cd}( \Alt_n)\setminus{\cd}(\Sym_n)$.
A lot of evidence for this conjecture had already been collected (in particular,
implying Theorem \ref{th:main} for some infinite series of values for $n$),
but the full result was only recently proved by K.\ Debaene \cite{Debaene}.
\end{remark}

The rest of the paper is organized as follows. In Section~\ref{sec:prelim}, we collect some useful results on character degrees of simple groups. In Section \ref{sec:Adegrees}, we present several technical results on character degrees of alternating groups which will be needed in subsequent 
sections. Section~\ref{sec:solvable} is devoted to verifying 
 Step~1. Steps 2 and 3 will be verified in Section \ref{sec:cf} and \ref{sec:Largedegrees}, respectively. Finally,
 in Section \ref{sec:HCproof} we prove Theorem \ref{thm:almostHC} and Theorem \ref{th:main}.

\section{Preliminaries}\label{sec:prelim}


For a finite group $G$,  we write $\pi(G)$ for the set of
all prime divisors of the order of~$G$. Denote by $p(G)$ the largest prime divisor of the order of~$G$. Let $\rho(G)$ be the set of all primes which divide some irreducible
character degree of~$G$. If $\cd(G)=\{d_0,d_1,\ldots,d_\ell\}$,  with $d_i<d_{i+1},0\leq i\leq \ell-1$,  then we define $d_i=d_i(G)$ for $1\leq i\leq \ell$. Then $d_i(G)$ is the $i^{th}$
smallest degree of the nontrivial character degrees of~$G$.
The largest character degree of~$G$ will be denoted by~$b(G)$, and
we let $k(G)$  denote the number of conjugacy classes of~$G$.
Furthermore, if $N\unlhd G$ and
$\theta\in {\Irr}(N)$,  then the inertia group
of $\theta$ in $G$ is denoted by $I_G(\theta)$.
The set of all irreducible constituents of
$\theta^G$ is denoted by ${\Irr}(G|\theta)$. A group $G$ is called an almost simple group with socle $S$ if $S\unlhd G\leq \Aut(S)$ for some nonabelian simple group $S$.

We need a couple of results from number theory. The first is called Bertrand's postulate; a proof can be found in \cite{Ramanujan}.

\begin{lemma}\emph{(Tschebyschef).}\label{lem2} If $m\geq 7$,  then there is at least one prime $p$ with $m/2<p\leq m$.
\end{lemma}

The following is an elementary result.
\begin{lemma}\label{lem4} Let $n\geq 5$ be an integer and let $p$ be a prime. If
the $p$-part of $n!$ is $p^{\nu}$,  then $\nu\leq {n}/{(p-1)}$.
\end{lemma}

Combining the Ito-Michler Theorem with the fact that $\chi(1)$ divides $|G|$ for all $\chi\in \Irr(G)$,
we have the following known result.
\begin{corollary}\label{cor2} If $S$ is a nonabelian simple group then $\rho(S)=\pi(S)$.
\end{corollary}

Note that every simple group of Lie type $S$ in characteristic $p$
(excluding the Tits group) has an irreducible character of degree
$|S|_p$,  which is the size of the Sylow $p$-subgroup of~$S$,  and is
called the \emph{Steinberg} character of $S$, denoted by
$\St_S$. Moreover, this character extends to $\Aut(S)$,  the full
automorphism group of $S$.

\begin{lemma}\emph{(\cite[Lemma~2.4]{Tong2}.)}\label{lem3}
 Let $S$ be a simple group of Lie type  in
characteristic~$p$ defined over a finite field of size~$q$. Assume that $S\neq \PSL_2(q),{}^2\mathrm{F}_4(2)'$. Then
there exist two irreducible characters $\chi_i$  of $S$, $i=1,2$, such that both $\chi_i$ extend to
$\Aut(S)$ with $1<\chi_1(1)<\chi_2(1)$ and $\chi_2(1)=|S|_p$. In particular, if $G$ is an
almost simple group with socle $S$,  where $S\neq \PSL_2(q),{}^2\mathrm{F}_4(2)'$,  then $|S|_p> d_1(G)$.
\end{lemma}

In Table \ref{Tab1}, for each sporadic simple group or the Tits group $S$,  we list the largest prime divisor of $|S|$ and the two irreducible characters of $S$ which are both extendible to $\Aut(S)$. In Table \ref{Tab2}, we list the two smallest nontrivial degrees of $\Aut(S)$ where $\Out(S)$ is nontrivial.

The following lemma will be useful in the last section.

\begin{lemma}\emph{(\cite[ Theorem~2.3]{Moreto}).}\label{lemmaMoreto} Let $N$
be a normal subgroup of a group $G$ and let $\theta\in\Irr(N)$ be
$G$-invariant. If $\chi(1)/\theta(1)$ is a power of a fixed prime
$p$ for every $\chi\in\Irr(G|\theta)$ then $G/N$ is solvable.
\end{lemma}

\section{Character degrees of the alternating groups}\label{sec:Adegrees}

Let $n$ be a positive integer.
We call $\lambda=(\lambda_1,\lambda_2,\dots, \lambda_r)$ a  partition of $n$,
written $\lambda\vdash n$,  provided $\lambda_i, i=1,\dots, r$
are integers, with $\lambda_1\geq
\cdots\geq \lambda_r>0$ and $\sum_{i=1}^r \lambda_i=n$.
We collect the same parts together and write
$\lambda=(\ell_1^{a_1},\ell_2^{a_2},\dots,\ell_k^{a_k})$,
with $\ell_i>\ell_{i+1}>0$ for $i=1,\dots, k-1;a_i\neq
0;$ and $\sum_{i=1}^k a_i\ell_i=n$.
It is well known that the irreducible complex characters of the
symmetric group $\Sym_n$ are parameterized by partitions of~$n$.
Denote by $\chi^\lambda$  the
irreducible character of $\Sym_n$ corresponding to the partition~$\lambda$.
The irreducible characters of
the alternating group $\Alt_n$ are then obtained by
restricting $\chi^\lambda$ to $ \Alt_n$. In fact,
$\chi^\lambda$ is still irreducible upon restriction to the alternating group $ \Alt_n$ if and only if
$\lambda$ is not self-conjugate. Otherwise, $\chi^\lambda$ splits
into two different irreducible characters
of $\Alt_n$ having the same degree.

\medskip

Based on results by Rasala \cite{Rasala} we deduce the following list of
minimal degrees for the alternating groups.

\begin{lemma}\label{lem:mindeg}

\noindent$(a)$ If $n\geq 15$,  then

$(1)$ $d_1( \Alt_n)=n-1;$

$(2)$ $d_2( \Alt_n)=\frac{1}{2}n(n-3);$

$(3)$ $d_3( \Alt_n)=\frac{1}{2}(n-1)(n-2);$ 

$(4)$ $d_4( \Alt_n)=\frac{1}{6}n(n-1)(n-5);$

\smallskip
\noindent $(b)$ If $n\geq 22$,  then

$(5)$ $d_5( \Alt_n)=\frac{1}{6}(n-1)(n-2)(n-3);$

$(6)$ $d_6( \Alt_n)=\frac{1}{3}n(n-2)(n-4);$

$(7)$ $d_7( \Alt_n)=n(n-1)(n-2)(n-7)/24;$

$(8)$ $d_8( \Alt_n)=(n-1)(n-2)(n-3)(n-4)/24$.

\smallskip
\noindent $(c)$ If $n\geq 43$,  then

$(9)$ $d_9( \Alt_n)= n(n-1)(n-4)(n-5)/12$;

$(10)$ $d_{10}( \Alt_n)= n(n-1)(n-3)(n-6)/8$;

$(11)$ $d_{11}( \Alt_n)= n(n-2)(n-3)(n-5)/8$;

$(12)$ $d_{12}( \Alt_n)= n(n-1)(n-2)(n-3)(n-9)/120$.

\end{lemma}

\begin{proof}
The first seven degrees were already deduced from
the list of minimal degrees for the symmetric groups in
\cite[Corollary~5]{Tong1}.
Similar arguments can be applied for the other degrees,
by using the list of minimal degrees for $\Sym_n$ up to
$d_{14}(\Sym_n)$; these are given in  Rasala \cite{Rasala}.
For $n\geq 22$, the next smallest degrees of
$\Sym_n$ are $d_j(\Sym_n)$, $j\in \{8,9,10,11\}$,
with polynomials as in $(8)-(11)$. For $n\geq 43$, the degree  $d_{12}(\Sym_n)$ is
given by the polynomial in (12),
and the next smallest degrees are

\begin{tabular}{ccl}
 $d_{13}( \Sym_n)$&= & $(n-1)(n-2)(n-3)(n-4)(n-5)/120$ \\
 $d_{14}( \Sym_n)$&= &  $n(n-1)(n-2)(n-4)(n-8)/30.$\\

\end{tabular}

All these character degrees for $\Sym_n$ are attained
(in the corresponding range) only at
non-symmetric partitions $\lambda=(\lambda_1,\lambda_2, \ldots)$, namely
at partitions where $n-\lambda_1 \leq 5$.
Thus they restrict to irreducible characters of $\Alt_n$.

For $n\geq 22$, we want to argue that $d_8(\Alt_n)$ is given as above.
Now a constituent $\chi \in \Irr(\Alt_n)$ in the restriction of a character
of degree $> d_{11}(\Sym_n)$  satisfies
$$\chi(1)>\frac 12 d_{11}(\Sym_n) > d_8(\Sym_n),$$ hence the formula in $(8)$ holds.

For $n\geq 43$, a constituent $\chi \in \Irr(\Alt_n)$ in the restriction of a character of degree $> d_{14}(\Sym_n)$  satisfies
$$\chi(1)>\frac 12 d_{14}(\Sym_n) > d_{12}(\Sym_n),$$ hence we have
the formulae in $(9)-(12)$. \end{proof}

The list above is crucial for excluding degrees that will come up in
later sections.

\begin{lemma}\label{lem:excl1}

Let $n\in \N$, $n\geq 14$, and let $s,t\in \N$.
\begin{enumerate}
\item[{(a)}]
If $n-1 = s(t-1)$ 
then $st \not\in \cd(\Alt_n)$.
\smallskip

\item[{(b)}] If $s>1$ and $n(n-3)/2 =s(t-1)$ 
then $st \not\in \cd(\Alt_n)$.
\end{enumerate}
\end{lemma}

\begin{proof}
(a) For $n=14$, we have $14, 26\not\in \cd(\Alt_n)$, so the claim holds.
When $n\geq 15$,  the inequality $$n-1< st = n-1+s \leq 2(n-1) < n(n-3)/2,$$
yields the assertion by Lemma~\ref{lem:mindeg}.

(b) For $n=14$, the assertion is checked directly,
so we may now assume $n\geq 15$.
If $t=2$, then $$(n-1)(n-2)/2< st=n(n-3) < n(n-1)(n-5)/6,$$  and if $t>2$ then $s\leq n(n-3)/4$ and  $$(n-1)(n-2)/2 < st = n(n-3)/2 + s \leq 3n(n-3)/4 < n(n-1)(n-5)/6,$$ hence part~(b) holds, again by Lemma~\ref{lem:mindeg}.
\end{proof}

\medskip

\begin{lemma}\label{lem:excl2}

Let $n\in \N$.
\begin{enumerate}
\item[{(a)}]
For any $1 < m \mid n-1$, $m(n-1) \not\in \cd(\Alt_n)$.
\smallskip

\item[{(b)}]
For $1 < s\mid n-1$, $s(n-1)(n-2)/2 \not\in \cd(\Alt_n)$ except when $s=2$, $n=9$ or $s=12$, $n=13$,
and   $s(n-1)(n-2)(n-3)/6 \not\in \cd(\Alt_n)$, except when
$s=3$ and $n\in\{4,10,16\}$.
\smallskip

\item[{(c)}] 
For $i\leq 3$, $n(n-1)(n-3)/2^i \in \cd(\Alt_n)$,
only when $i=1$ and $n\in \{9,10,14\}$, or $i=2$ and $n\in \{4,8,12\}$,
or $i=3$ and $n\in \{5,7,8,11\}$.
\smallskip

\item[{(d)}] 
For $i\leq 3$, $n(n-1)(n-2)(n-4)/3^i \in \cd(\Alt_n)$ only when
$i=1$ with $n\in\{11,12,13,18,23\}$, or $i=3$ with $n\in \{10,13,27,31\}$.

\item[{(e)}] 
For $n=14$, $7280=n(n-1)(n-2)(n-4)/3  \notin \cd(\Alt_n)$.
For $n=16$ and $s=3$ or~$5$,
$s\binom{15}{5}\notin \cd(\Alt_{n})$.
\end{enumerate}
\end{lemma}

\begin{proof}
Set $d_j=d_j(\Alt_n)$ for $1\le j\in\N$.

(a) For $n<15$, $(n-1)^2$ is not a degree (by inspection),
and for $n\geq 15$,
$d_3<(n-1)^2<d_4$ shows the assertion by Lemma~\ref{lem:mindeg}.


\smallskip

(b) Let $1 < s\mid n-1$. Set $t=(n-1)(n-2)/2$.

Assume first that $s=n-1$.
For $n<22$, $st$ is in $\cd(\Alt_n)$ only for $n=13$.
For $n\ge 22$, we have $d_6< st < d_7$, and thus  $st$ is
not in $\cd(\Alt_n)$.

Assume now that $1<s<n-1$, so $2\le s \le (n-1)/2$.
For $n<22$, $st$ is in $\cd(\Alt_n)$ only for $n=9$, $s=2$.
For $n\ge 22$, we have $d_3 < st < d_6$, so we only have to check
that $st\ne d_4, d_5$.
Now $st=d_4$ implies $3s(n-2)=n(n-5)$ and hence $n\mid 3(n-2)$, giving $n\mid 6$,
a contradiction.
If $st=d_5$, then $3s=n-3$, and thus $s\mid (n-1,n-3)\leq 2$ and $n\le 9$, again a contradiction.

For the second assertion, we now set $u=(n-1)(n-2)(n-3)/6$.
For $n<43$, we only find the stated exceptional cases (by computation),
so we now assume $n\ge 43$.

Again, we start with the case $s=n-1$.
We deduce from
$d_{11} < su < d_{12}$
that $su \not\in\cd(\Alt_n)$.

Now consider $1<s<n-1$, i.e.,  $2\le s \le (n-1)/2$.
Then we have $d_6 < su < d_{10}$, and we only have to show
that $su \not\in\{d_7,d_8,d_9\}$.
An easy consideration of cases similar to the previous case shows that
no further exception arises.
\smallskip

(c)
Let $t_i=n(n-1)(n-3)/2^i$.

For $n\leq 21$ we find the exceptions with a computation (by Maple, say).
So we may assume $n\ge 22$.
One easily checks: $d_6 < t_1 < d_7$, $d_5 < t_2 < d_6$,
and $d_3<t_3<d_4$.
Hence $t_1, t_2, t_3 \not\in \cd(\Alt_n)$ for $n\ge 22$.
\smallskip

(d)
Let $t_i=n(n-1)(n-2)(n-4)/3^i$.

For $n\leq 42$ we find the stated exceptions by computation.
So we may assume $n\ge 43$.
One easily checks: $d_{11} < t_1 < d_{12}$, $d_9 < t_2 < d_{10}$,
and $d_6<t_3<d_7$.
Hence $t_1, t_2, t_3 \not\in \cd(\Alt_n)$ for $n\ge 43$.
\smallskip

(e) is easily checked directly.
\end{proof}

\section{Solvable groups}\label{sec:solvable}

First we collect some preliminary facts.

\begin{lemma}\label{lem:arithmetic}
Let $n\in \N$.    
Then the following holds.
\begin{enumerate}
\item[{(a)}] $(n-1, n(n-3)/2) =d>1$ if and only if $n=4k+3$ and $d=2$.
\smallskip

\item[{(b)}] Let $s, n, a, b\in \N$ and $p$ a prime with $s(n-1)=p^a-1$ and $p^b\mid n$ where $b=a$ or~$a/2$.
Then $n=p^a$ or $p^b$.
\end{enumerate}
\end{lemma}

\begin{proof}
(a) is easy arithmetic.
\smallskip

(b)
Let $t$  be such that $tp^b=n$, then $s(tp^b-1)=p^a-1$. So $st p^b=p^a+s-1$.
There exists a non-negative integer $k$ satisfying $s-1=kp^b$.
Now $t(kp^b+1)p^b=p^a+kp^b$, $ts=p^{a-b}+k$.
If $a=b$ then $t kp^a =t(s-1)= 1+k-t\leq k$ which implies that $k=0$
and $s=1$ with $n=p^a$. If $a=2b$  then $(kt-1)p^b=k-t$, so either $t=p^b$
and $n=p^a$, or $t=1$ and $n=p^b$.
\end{proof}

\begin{lemma}\label{lem:group}
Let $n\in \N$, $n\ge 14$,
and $G$ a solvable group with $\cd(G)=\cd(\Alt_n)$.
Then the following holds.
\begin{enumerate}
\item[{(a)}]
$G$ has no normal subgroup $N$ with $G/N$  nonabelian of prime power order.
\smallskip

\item[{(b)}]
For  $\phi \in \Irr(G)$ such that $\phi(1)=n-1$ or $n(n-3)/2$,
if $\phi(1)$ is the minimal nonlinear degree in $\cd(G/\Kernel(\phi))$
and $\Kernel(\phi )$  is maximal possible under subgroup inclusion within~$G$,
then $G/\Kernel(\phi )$ is a Frobenius group with a minimal normal subgroup
as the Frobenius kernel.
\end{enumerate}

\end{lemma}

\begin{proof}
(a)  Suppose that $G/N$  is a nonabelian $p$-group and let $\phi $
be a nonlinear  irreducible character of $G$ with $N\leq \Kernel(\phi )$,
then $n-1=p^b =\phi (1)$ by \cite[Theorem 5.1]{BBOO}.  For $\chi \in \Irr(G)$ with $\chi(1)=n(n-3)/2$
the restriction of  $\chi $ to $N$ is irreducible, thus for any
$\beta \in \Irr(G/N), \chi \beta \in \Irr(G)$.
In particular $(n-1)n(n-3)/2 \in \cd(\Alt_n)$, by Lemma~\ref{lem:excl2}
we have $n=14$, but then there exists  $\chi' \in \Irr(G)$ with
$\chi'(1)=560$, however $(n-1)\chi'(1)\notin \cd(\Alt_n)$, a contradiction.
\smallskip

(b) By  \cite[Theorem 1]{Robinson} and (a), if  $G/\Kernel(\phi)$ is not a Frobenius group
then $\phi(1)=s(t-1)$ with $t-1$
a prime power and $st\in \cd(G/\Kernel(\phi))\subseteq \cd(\Alt_n)$,
this contradicts Lemma~\ref{lem:excl1}.
\end{proof}

\bigskip

\begin{theorem}\label{thm:nonsolvable}
Let $n\in \N$, $n\geq 14$.
If $G$ is a finite group with $\cd(G)=\cd(\Alt_n)$
then $G$ is nonsolvable.
\end{theorem}

\begin{proof}
Suppose toward a contradiction that $G$ is solvable.  Let $\phi$ be an irreducible character of $G$ of degree $n-1$ such that $\Kernel(\phi)$ is of maximal order among the kernels of irreducible characters of $G$ of the same degree.
By Lemma~\ref{lem:group} and \cite[Theorem 1]{Robinson} we see that $G/L=\overline{N}\cdot \overline{H}$ is a Frobenius group, where $L=\Kernel(\phi)$, $N$ and $H$ are subgroups of $G$ containing $L$ such that $\overline{N}$ is a minimal normal subgroup   of $G/L$ of order $p^a$ and $\overline{H}$ is the complement of $\overline{N}$ and is cyclic of order $n-1$.

We first consider the case  $n=14$.  Choose $\eta \in \Irr(G) $ such that $\eta(1)=560=2^4\cdot 5\cdot 7$, then $\eta_N \in \Irr(N)$, by \cite[Theorem 6.18]{Isaacs} we have  $\eta_L=\eta_1, e\eta_1 $ or $\eta_1+ \cdots+ \eta_{p^a}$
where $\eta_i\in \Irr(L)$ and $e^2=p^a$. Since we see that $e$ does not divide $560$ (if $p\mid 70$ then $a=12, e=6$),  $\eta_L=\eta_1$ then $13\cdot 560\in \cd(\Alt_{14})$, a contradiction.

So assume $n> 14$. Let $\chi$ be another irreducible character of $G$ of degree $n(n-3)/2$.
Let $\beta \in \Irr(N)$ be an irreducible constituent of the restriction $\chi_N$. Then $\beta$ extends to $\beta'\in \Irr(I_G(\beta))$.  We claim that $I_G(\beta)$ acts irreducibly on $\overline{N}$.  If $I_G(\beta)=G$ the claim is obvious. So we need only consider the case where $I_G(\beta)$ is a proper subgroup of $G$. Note that
since $|G/I_G(\beta)|$ divides both $n-1$ and $\chi(1)$,
$|G/I_G(\beta)|=(n-1,n(n-3)/2)=2$ with $n=4k+3$.
Let $x\in H$ be such that  $\la \overline{x}\ra = \overline{H}$; then $x^2\in I_G(\beta)$ and $\overline{x}^{2k+1}$ inverses every element of $\overline{N}$, which implies that $x^2$ acts irreducibly on $\overline{N}$, so the claim is true.

Now by \cite[Theorem 6.18]{Isaacs} for $I_G(\beta)$ and $\beta $ we have $\beta_L=\theta_1, e\theta_1$ or $\Sigma_{i\leq p^a}\theta_i$, where $p^a=|\overline{N}|$, $e^2=p^a$ and
$\theta_i\in \Irr(L)$. If $\beta_L=\theta_1$ then $I_G(\beta )\leq I_G(\theta_1)$ and for any
$\alpha \in \Irr(I_G(\beta)/L), \alpha\beta'\in \Irr(I_G(\beta))$.
Note that
$$ \alpha(1)\beta'(1)\leq \frac{n-1}{d}\cdot\frac{n(n-3)}{2d}$$
where $d=|G/I_G(\beta)|$. Since we can choose $\alpha(1)=(n-1)/d$, either $n(n-1)(n-3)/2$,  or ${n(n-1)(n-3)}/{2d}$ or
${n(n-1)(n-3)}/{2d^2}$  lies in $\cd(\Alt_n)$, which is impossible.
Thus we have either $p^a\mid \chi(1)$ or $a=2b$ with $p^b\mid \chi(1)$.
By Lemma~\ref{lem:arithmetic}(b), if $p^b\mid n$
we have  $n=p^b$ or $p^a$, thus $n=p^a$ as $\overline{N}$ minimal in $G/L$.  For the case where $p^b$ does not divide $n$, we claim that either $p\mid n$
or $n=2p^b +3 $. Suppose $(p, n)=1$, then $p^b\mid(n-3)$.
If $b=a$, then $p^a-1 \leq n-2 < n-1$, this is contradictory to $p^a-1$ divisible by $n-1$. So $a=2b$ with $n-3= t'p^b$ and $p^{2b}-1=s'(n-1)$ for some positive integers $s'$ and $t'$. Since
$$s't'p^b =s'(n-3)=s'(n-1)- 2s'=p^{2b}-1 -2s',$$
we have $2s'+1=kp^b$ for some positive integer $k$. Thus
$$2p^b=2s't'+2k=t'(kp^b-1)+2k.$$
It follows that $(t'k-2)p^b=t'-2k$, which has as its only solution $t'=2$ with $k=1$, so $n=2p^b +3 $, and the claim is true. In particular, if $(p, n)=1$ then
$n\not=  18, 23, 27 $ or $31$, and if $n=p^c$ for some positive integer $c$ then $c=a$.
In any case, we can exclude $n=18$:  when $p\mid n=18$, either $p=2$ with $a=2b=8$ or $p=3$ with $a=2b=16$, from which we see that $p^b$ does not divide $18(18-3)/2 = 3^3\cdot 5$.

Now choose $\zeta \in \Irr(G)$ with $\zeta(1)=n(n-2)(n-4)/3$. Let $\sigma\in \Irr(N)$ be an irreducible constituent of $\zeta_N$. We claim that $I_G(\sigma)$ acts irreducibly on $\overline{N}$. Suppose this is not the case, then $I_G(\sigma)$ is bound to be a proper subgroup of $G$ and acts reducibly on $\overline{N}$. So $\zeta_N=\sigma + \sigma_2 + \sigma_3$ and $3=(n-1, n(n-2)(n-4)/3)$ with $n=9k+4$, where $\sigma, \sigma_2$ and $\sigma_3$ are conjugate irreducible characters of $N$.  As
$n-1=9k+3=3(3k+1), x^3\in I_G(\sigma)$,
$$\overline{N}=V_1 \oplus V_2 \oplus V_3$$
where the $V_i$'s are $\overline{x}^3$-spaces and thus $3\mid a$.  Let $C$ be the centralizer of $\overline{x}$ in $\GL(\overline{N})$; we see that $C$ is cyclic of order $p^a-1$ and conjugate in $\GL(\overline{N})$ to the multiplicative group $\F_{p^a}^{*}$ of the Galois field $\F_{p^a}$ ($\overline{N}$ is viewed as the additive group of  $\F_{p^a}$), so the action of
$\la \overline{x}\ra$ on $\overline{N}$ is just the multiplication by elements of  $\F_{p^a}^{*}$.  Now $\overline{N}$ can be viewed as a space over $\F_{p^s}$ where $s=1$ if $3\mid (p-1)$ and $s=2$ if $3\mid (p+1)$, so $\overline{x}$ can be viewed as an element of $\GL(m', p^s)$ where $m'$ is the dimension of $\overline{N}$  over $\F_{p^s}$.  Note that each $V_i$ is of dimension divisible by $s$
and $\la\overline{x}^{3k+1}\ra$ is the only subgroup of order $3$ in $\F_{p^a}^{*}$,  thus $V_i$ is  also  an $\la\overline{x}\ra$-space, a contradiction. Hence the claim holds true.  Evidently $\sigma$ extends to $\sigma'\in \Irr(I_G(\sigma))$.   By \cite[Theorem 6.18]{Isaacs} for $I_G(\sigma)$ and $\sigma $ we have $\sigma_L=\sigma_1, e\sigma_1$ or $\Sigma_{i\leq p^a}\sigma_i$, where $e^2=p^a$ and
$\sigma_i\in \Irr(L)$.  We claim that $n=p^a$. Suppose otherwise that $n\not=p^a$, then as discussed above $p^b$ does not divide $n$ and $n\not= 18, 23, 27 $ or $31$. If $\sigma_L=\sigma_1$ then $I_G(\sigma )\leq I_G(\sigma_1)$ and for any
$\alpha' \in \Irr(I_G(\sigma)/L), \alpha'\sigma'\in \Irr(I_G(\sigma))$. Note that \[\alpha'(1)\sigma'(1)\leq \frac{n-1}{d'}\frac{n(n-2)(n-4)}{3d'},\] where $d'=|G/I_G(\sigma)|$. Since we can choose $\alpha'(1)=(n-1)/d'$, one of $$\frac{n(n-1)(n-2)(n-4)}{3}, \: \frac{n(n-1)(n-2)(n-4)}{3d'},\: \frac{n(n-1)(n-2)(n-4)}{3d'^2}$$
lies in $\cd(\Alt_n)$, which is impossible by
Lemma~\ref{lem:excl2}(d). Thus we have either $p^a\mid \zeta(1)$ or $a=2b$ with $p^b\mid \zeta(1)$. Now we have $p^b\mid (\chi(1), \zeta(1))$ and $p^b$ does not divide $n$, so $p\mid (n-3, (n-2)(n-4))=1$,  which is absurd. Thus we have $n=p^a$ as claimed.

Choose $\rho_i \in \Irr(G)$ such that $\rho_i(1)=\binom{n-1}{i}$, $i=2, 3$
or $5$, then we have $$(\rho_i)_N=\zeta_i+\zeta_{i2}+ \cdots + \zeta_{it}$$
where $t=|G/I_G(\zeta_i)|$.
Note that $\zeta_i$ extends to $\zeta_i' \in \Irr(I_G(\zeta_i))$.
Since $(p, \rho_i(1))=1$, $(\zeta_i)_L=\delta_i \in \Irr(L)$.
Thus for any $\lambda \in \Irr(N/L), \lambda \zeta_i\in \Irr(N)$ and
$\lambda \zeta_i \not= \lambda' \zeta_i$ if $\lambda \not=\lambda' $.
We claim that $t=n-1$. Suppose that $t< n-1$ then $I_G(\zeta_i)/N$
is cyclic of order ${(n-1)}/{t} > 1$.
If $I_G(\zeta_i)=I_G(\delta_i )$ then for
$\alpha_i \in \Irr(I_G(\zeta_i)/L)$ with $\alpha_i(1)={(n-1)}/{t}$,
$(\alpha_i\zeta_i')^G \in \Irr(G)$ and is of degree
${(n-1)}\rho_i(1)/t\in \cd(\Alt_n)$, which is impossible
(for $ i= 2$  with $n$ arbitrary or  $i=3$ with $n \not= 16$ or $i=5$
with $n= 16$).  So $I_G(\zeta_i)<I_G(\delta_i )$.
Now for any $y\in I_G(\delta_i)\setminus I_G(\zeta_i)$
(of course we choose $y\in \la x \ra $), there is
a nontrivial linear character $\lambda \in \Irr(N/L)$ such that  $\zeta_i^y=\lambda \zeta_i$.
It follows that for any $v\in I_G(\zeta_i)\setminus N$
with $v\in \la x \ra$,  $$\lambda \zeta_i=(\zeta_i)^{vy}= (\zeta_i)^{yv}= (\lambda \zeta_i)^v=\lambda^v\zeta_i,$$ so  $\lambda =\lambda^v$, contradictory to the fixed point free action of $ I_G(\zeta_i)/N$ on $\Irr(N/L)$. Thus $t=n-1$ as claimed. It follows $I_G(\zeta_i)=N$, and $\zeta_i(1)={\rho_i(1)}/{(n-1)}$ which implies $n\not=16$ (as $\binom{n-1}{5}/(n-1)$ is not an integer for $n=16$)  and $2\mid (p^a-2)$ with $6\mid (p^a-2)(p^a-3)$, so $p=2$ and $a=2k+1$. As discussed above, for  $\chi$ and $\beta$,  $|G/I_G(\beta)|=(n-1, n(n-3)/2)=1$,  $G=I_G(\beta)$ and $\beta=\chi_N$. Note that $p^a$ is not a square, $\beta_L=\Sigma_{i\leq 2^a}\theta_i$, thus $\theta_1(1)=(2^a-3)/2$ which is absurd.  We are done.
\end{proof}

\section{Nonabelian composition factors}\label{sec:cf}

Recall that a group $G$ is said to be an {\em almost simple group}
if there exists a nonabelian simple group $S$ such that $S\unlhd
G\leq \Aut(S)$.
In this section, we show that every nonabelian chief
factor of a finite group $G$ with $\cd(G)=\cd( \Alt_n),n\geq 14$,  is isomorphic to $ \Alt_n$.

\begin{table}
 \begin{center}
  \caption{Sporadic simple groups and the Tits group} \label{Tab1}
  \begin{tabular}{lllr|lllr}
   \hline
   S &$p(S)$ & $\theta_i$ & $\theta_i(1)$ &S&$p(S)$& $\theta_i$&$\theta_i(1)$\\ \hline
   $\textrm{M}_{11}$& $11$ &$\chi_5$&$11$& $\textrm{O'N}$&$31$&$\chi_2$&$2^6\cdot 3^2\cdot 19$\\

   &&$\chi_6$&$2^4$&&&$\chi_{7}$&$2^7\cdot 11\cdot 19$\\

   $\textrm{M}_{12}$&$11$&$\chi_6$&$3^2\cdot 5$&$\textrm{Co}_3$&$23$&$\chi_2$&$23$\\
   &&$\chi_7$&$2\cdot 3^3$&&&$\chi_5$&$5^2\cdot 11$\\

   $\textrm{J}_1$&$19$&$\chi_2$&$2^3\cdot 7$&$\textrm{Co}_2$&$23$&$\chi_2$&$23$\\

   &&$\chi_4$&$2^2\cdot 19$&&&$\chi_{4}$&$5^2\cdot 11$\\

   $\textrm{M}_{22}$&$11$&$\chi_2$&$3\cdot 7$&$\textrm{Fi}_{22}$&$13$&$\chi_{2}$&$2\cdot 3\cdot 13$\\
   &&$\chi_3$&$3^2\cdot 5$&&&$\chi_{3}$&$3\cdot 11\cdot 13$\\

   $\textrm{J}_2$&$7$&$\chi_6$&$2^2\cdot 3^2$&$\textrm{HN}$&$19$&$\chi_{4}$&$2^3\cdot 5\cdot 19$\\
   &&$\chi_{7}$&$3^2\cdot 7$&&&$\chi_{5}$&$2^{4}\cdot 11\cdot 19$\\

   $\textrm{M}_{23}$&$23$&$\chi_2$&$2\cdot 11$&$\textrm{Ly}$&$67$&$\chi_3$&$2^4\cdot 5\cdot 31$\\
   &&$\chi_3$&$3^2\cdot 5$&&&$\chi_{4}$&$2\cdot 11\cdot 31\cdot 67$\\

   $\textrm{HS}$&$11$&$\chi_2$&$2\cdot 11$&$\textrm{Th}$&$31$&$\chi_2$&$2^3\cdot 31$\\
   &&$\chi_3$&$7\cdot 11$&&&$\chi_3$&$7\cdot 19\cdot 31$\\

   $\textrm{J}_3$&$19$&$\chi_6$&$2^2\cdot 3^4$&$\textrm{Fi}_{23}$&$23$&$\chi_2$&$2\cdot 17\cdot 23$\\
   &&$\chi_{9}$&$2^4\cdot 3\cdot 17$&&&$\chi_{3}$&$2^{2}\cdot 3\cdot 13\cdot 23$\\

   $\textrm{M}_{24}$&$23$&$\chi_2$&$23$&$\textrm{Co}_1$&$23$&$\chi_2$&$2^2\cdot 3\cdot 23$\\
   &&$\chi_3$&$3^2\cdot 5$&&&$\chi_{3}$&$13\cdot 23$\\

   $\textrm{McL}$&$11$&$\chi_2$&$2\cdot 11$&$\textrm{J}_4$&$43$&$\chi_3$&$31\cdot 43$\\
   &&$\chi_{3}$&$3\cdot 7\cdot 11$&&&$\chi_{4}$&$3^2\cdot 29\cdot 31\cdot 37$\\

   $\textrm{He}$&$17$&$\chi_6$&$2^3\cdot 5\cdot 17$&$\textrm{Fi}_{24}'$&$29$&$\chi_2$&$13\cdot 23\cdot 29$\\
   &&$\chi_{9}$&$3\cdot 5^2\cdot 17$&&&$\chi_3$&$3\cdot 7^2\cdot 17\cdot 23$\\

   $\textrm{Ru}$&$29$&$\chi_2$&$2\cdot 3^3\cdot 7$&$\textrm{B}$&$47$&$\chi_2$&$3\cdot 31\cdot 47$\\
   &&$\chi_{4}$&$2\cdot 7\cdot 29$&&&$\chi_{3}$&$3^{3}\cdot 5\cdot 23\cdot 31$\\

   $\textrm{Suz}$&$13$&$\chi_2$&$11\cdot 13$&$\textrm{M}$&$71$&$\chi_2$&$47\cdot 59\cdot 71$\\
   &&$\chi_{3}$&$2^{2}\cdot 7\cdot 13$&&&$\chi_{3}$&$2^2\cdot 31\cdot 41\cdot 59\cdot 71$\\

   ${}^2\textrm{F}_4(2)'$&$13$&$\chi_5$&$3^3$&$$&$$&$$&$$\\
   &&$\chi_{6}$&$2\cdot 3\cdot 13$&&&$$&\\
   \hline
\end{tabular}
\end{center}
\end{table}

\begin{table}
 \begin{center}
  \caption{Automorphism groups of sporadic simple groups}\label{Tab2}
  \begin{tabular}{l|crr}
   \hline
   $G$  & $p(G)$ & $d_1(G)$&$d_2(G)$\\ \hline
   $\textrm{M}_{12}\cdot 2$&$11$&$22$&$32$\\
   $\textrm{M}_{22}\cdot 2$&$11$&$21$&$45$\\
   $\textrm{J}_2\cdot 2$&$7$&$28$&$36$\\
   $\textrm{HS}\cdot 2$&$11$&$22$&$77$\\
   $\textrm{J}_3\cdot 2$&$19$&$170$&$324$\\
   $\textrm{McL}\cdot 2$&$11$&$22$&$231$\\
   $\textrm{He}\cdot 2$&$17$&$102$&$306$\\
   $\textrm{Suz}\cdot 2$&$13$&$143$&$364$\\
   $\textrm{O'N}\cdot 2$&$31$&$10944$&$26752$\\
   $\textrm{Fi}_{22}\cdot 2$&$13$&$78$&$429$\\
   $\textrm{HN}\cdot 2$&$19$&$266$&$760$\\
   $\textrm{Fi}_{24}'\cdot 2$&$29$&$8671$&$57477$\\
   ${}^2\textrm{F}_4(2)'\cdot 2$&$13$&$27$&$52$\\
   \hline
  \end{tabular}

 \end{center}
\end{table}

\begin{theorem}\label{th:composition} Let $G$ be a group such that $\cd(G)=\cd( \Alt_n)$ with $n\geq 14$. If
$L/M$ is a nonabelian chief factor of $G$ then $L/M\cong  \Alt_n$.
\end{theorem}

\begin{proof} Assume that $L/M$ is a nonabelian chief factor of $G$. Then $L/M\cong S^k$,  where $k\geq 1$ and
$S$ is a nonabelian simple group. Let $C$ be a normal subgroup of $G$ such that $C/M=\Centralizer_{G/M}(L/M)$.
Then $LC/C\cong S^k$ is a unique minimal normal subgroup of $G/C$ so
that $G/C$ embeds into ${{\Aut}}(S)\wr \Sym_k$,  where $\Sym_k$ is the
symmetric group of degree~$k$. Let $B={{\Aut}}(S)^k\cap G/C$. Then
$|G/C:B|$ is isomorphic to a transitive subgroup of $\Sym_k$. Suppose
that $\theta\in \Irr(S)$ such that $\theta$ extends to $\Aut(S)$.
Let $\psi=\theta\times 1\times\cdots\times 1$ and $\varphi=\theta^k$
be irreducible characters of $LC/C\cong S^k$. By the character
theory of wreath products, $\varphi$ extends to
$\varphi_0\in\Irr(G)$ so $\theta(1)^k\in \cd(G)$. Let $I$ be the
inertia group of $\psi$ in $G$. Then $\psi$ extends to $\psi_0\in
\Irr(I)$ and hence $k\psi_0(1)=k\theta(1)\in\cd(G)$.
It follows from Corollary \ref{cor2} that if $r$ is any prime
divisor of $|S|$,  then there exists $\phi\in\Irr(S)$ such that
$r\mid \phi(1)$. Let $\gamma=\phi^k\in\Irr(LC/C)$. As $LC/C\unlhd
G/C$,  we deduce that $\gamma(1)=\phi(1)^k$ must divide $\chi(1)$ for
some $\chi\in\Irr(G)$ by \cite[Lemma~6.8]{Isaacs}. As
$\cd(G)=\cd( \Alt_n)$,  we obtain that $\chi(1)$ divides $| \Alt_n|$,  which
implies that $r^k\mid {n!}/{2}$. Recall that $p(S)$ is the largest prime divisor of $S$ and since $|\pi(S)|\ge 3$ we have $p(S)\ge 5$. Again, set $d_j=d_j(G)$ for $j\ge 1$.

\medskip
{\bf Claim 1.} $k=1$. Suppose that $k\geq 2$. 
By the discussion above,  $p(S)^k$ divides ${n!}/{2}$,  so by Lemma \ref{lem4} we have
\begin{equation}\label{eq1}
    2\leq k\leq \dfrac{n}{p(S)-1}.
\end{equation}

Observe that if $p$ is any prime and ${n}/{2}<p\leq n$,  then $| \Alt_n|_p=p$. Since $k\geq 2$
and $p(S)^k$ divides $| \Alt_n|$,  we deduce that $p(S)\leq {n}/{2}$ and thus

\begin{equation}\label{eq2}
    n\geq 2p(S).
\end{equation}

Using the classification of finite simple groups, we consider the following cases.

\medskip
 {\bf (1a)} $S$ is a sporadic simple group or the Tits group. Using \cite{atlas}, for each
sporadic simple group or the Tits group $S$,  there exist two nontrivial irreducible characters
$\theta_i\in\Irr(S)$ such that both $\theta_i$ extend to $\Aut(S)$ and $11\leq
\theta_1(1)<\theta_2(1)$ (see Table \ref{Tab1}). By the argument above, we obtain that
$k\theta_i(1)\in\cd(G)$ for $i=1,2$. Using \cite{GAP} for $n=14$ and  Lemma~\ref{lem:mindeg}$(a)$ for $n\geq 15$,  we have

$$\begin{array}{lll}
    d_1& =& n-1 \\
    d_2 & =&n(n-3)/2 \\
    d_3 & =&(n-1)(n-2)/2 \\
    d_4 & =&n(n-1)(n-5)/6.
  \end{array}
$$

We first claim that
\begin{equation}\label{eqn}
k\theta_1(1)\leq d_3=\frac{(n-1)(n-2)}{2}.
\end{equation}
As $n\geq 2p(S)$,  by checking Table~\ref{Tab1} we obtain that \[\frac{(n-1)(n-5)}{6}\geq \frac{(2p(S)-1)(2p(S)-5)}{6}>\frac{\theta_1(1)}{p(S)-1}.\]
Since $k\leq {n}/{(p(S)-1)}$,  we have \[d_4=\frac{n(n-1)(n-5)}{6}>\frac{n\theta_1(1)}{p(S)-1}\geq k\theta_1(1).\]
Thus $k\theta_1(1)<d_4$ and so $k\theta_1(1)\le d_3$ which proves our claim.
As $k\geq 2$ and $d_3=(n-1)(n-2)/2$,  we obtain

\begin{equation}\label{eq3}
\frac{(n-1)(n-2)}{2}\geq 2\theta_1(1).
\end{equation}
We now consider the following cases.

\smallskip
$(i)$
$S\in\{\textrm{M}_{11},\textrm{M}_{12},\textrm{J}_1,\textrm{M}_{22},\textrm{M}_{23},\textrm{HS},\textrm{M}_{24},\textrm{Ru},{}^2\textrm{F}_4(2)',\textrm{Co}_1,\textrm{Co}_2,\textrm{Co}_3\}$.

Since $n\geq 2p(S)$,  for all simple groups in this case we can check that
$$\frac{\theta_2(1)}{p(S)-1}<\frac{2p(S)-3}{2}\le\frac{n-3}{2}.$$
From \eqref{eq1}, we deduce that $$k\theta_2(1)\leq
\frac{n\theta_2(1)}{p(S)-1}< \frac{n(n-3)}{2}=d_2.$$ But this is impossible as $k\theta_2(1)>k\theta_1(1)\geq d_1$ so $k\theta_2(1)\geq d_2$.

\smallskip
$(ii)$
$S\in\{\textrm{J}_2,\textrm{J}_3,\textrm{McL},\textrm{He},\textrm{Suz},\textrm{Fi}_{22},\textrm{HN},\textrm{Ly},\textrm{Th},\textrm{Fi}_{23},\textrm{Fi}_{24}',\textrm{B}\}$.

Since $n\geq 2p(S)$,  we
deduce  that $$k\theta_2(1)\leq
\frac{n\theta_2(1)}{p(S)-1}<\frac{n(n-1)(n-5)}{6}=d_4$$ unless $S\in \{\textrm{Fi}_{24}',\textrm{B}\}$. For the exceptions,  applying
\eqref{eq3}, we have that $n\geq 134$ when $S\cong \textrm{B}$ and
$n\geq 188$ when $S\cong \textrm{Fi}_{24}'$. For these cases, we also obtain that $k\theta_2(1)<d_4$.   Thus $k\theta_2(1)\leq d_3$.
Observe that $(d_2,d_3)=1$ and  $(d_1,d_2)\leq 2$. For
each sporadic simple group $S$ in this case, we can check that
$(\theta_1(1),\theta_2(1))\geq 2$ and so
$(k\theta_1(1),k\theta_2(1))\geq 4$. Since $d_1\leq k\theta_1(1)<k\theta_2(1)\leq d_3$ and $(k\theta_1(1),k\theta_2(1))\geq 4$,  we have $k\theta_1(1)=d_1=n-1$ and $$k\theta_2(1)=d_3=\frac{1}{2}(n-1)(n-2)=\frac{1}{2}k\theta_1(1)(n-2).$$
Hence
 $2\theta_2(1)=(n-2)\theta_1(1)$. In particular,  $\theta_1(1)$ divides $2\theta_2(1)$ and $$n=\frac{2\theta_2(1)}{\theta_1(1)}+2.$$ Inspecting Table
\ref{Tab1}, we deduce that $S\cong \textrm{McL}$ or $S\cong \textrm{Fi}_{22}$. If $S\cong \textrm{McL}$,  then $n=23$. But
then $n-1=22=k\theta_1(1)$,  which implies that $k=1$,  a
contradiction. Similarly, if $S\cong \textrm{Fi}_{22}$,  then
$n=13<14$,  a contradiction.

\smallskip
$(iii)$ $S\in \{\textrm{O'N},\textrm{J}_4,\textrm{M}\}$. Firstly, by applying \eqref{eq3}, we have that $n\geq 889$
when $S\cong \textrm{M}$ and $n\geq 211$ when $S\cong \textrm{O'N}$. Also by \eqref{eq2}, we have that
$n\geq 86$ when $S\cong \textrm{J}_4$. Observe that for each simple group $S$ in this case, we have
$$k\theta_2(1)\leq \frac{n\theta_2(1)}{p(S)-1}<d_7=\frac{n(n-1)(n-2)(n-7)}{24}.$$
Thus $k\theta_2(1)\leq d_6 =n(n-2)(n-4)/3$. If $k\theta_2(1)\leq d_3$,  then we can argue as in case $(ii)$ to obtain
a contradiction. Thus we can assume that $$k\theta_2(1)\geq d_4=\frac{n(n-1)(n-5)}{6}.$$ Combining with \eqref{eqn}, we obtain that
\begin{equation}\label{eq4}
n(n-5)\theta_1(1)\leq 3\theta_2(1)(n-2).
\end{equation}
If $S\cong \textrm{O'N}$ or $\textrm{M}$,  then we can check that \eqref{eq4} cannot happen. Assume $S\cong
\textrm{J}_4$. If $k\theta_1(1)=d_1=n-1$,  then $n\geq 1+2\theta_1(1)=2667$. But then
\eqref{eq4} cannot happen. Thus $k\theta_1(1)\geq d_2=n(n-3)/2$. Hence $n(n-3)/2\leq
n\theta_1(1)/42$,  which implies that $n\leq 66$,  a contradiction.

\medskip
{\bf (1b)}  $S\cong  \Alt_m$,  $m\geq 7$. Let $\chi_i\in\Irr(\Sym_m),1\le i\le 3$,  be
irreducible characters of $\Sym_m$  labeled by the partitions $(m-1,1),
(m-2,2)$ and $(m-2,1^2)$,  respectively. As these partitions are not
self-conjugate, we deduce that for all $i$,  $\chi_i$ are still
irreducible upon restriction to $ \Alt_m$. Let $\theta_i\in\Irr(S)$ be
the restrictions of $\chi_i$ to $ \Alt_m$. Then $\theta_i\in\Irr(S)$ are
all extendible to $\Aut(S)\cong \Sym_m$. By the Hook formula, we obtain
that $\theta_1(1)=m-1$,  $\theta_2(1)=m(m-3)/2$ and
$\theta_3(1)=\theta_2(1)+1=(m-1)(m-2)/2$. By Lemma \ref{lem2}, there
is a prime $p$ such that $m/2<p\leq m$ for each $m\geq 7$. Hence
$p(S)>m/2$,  where $p(S)$ is the largest prime divisor of the order
of $S\cong  \Alt_m$. By \eqref{eq1}, we have that
$$k\leq \frac{n}{p(S)-1}<\frac{n}{m/2-1}=\frac{2n}{m-2}.$$

Assume first that $k\theta_1(1)\geq
d_2=n(n-3)/2$. As $(m-1)/(m-2)<2$,  we have
$$\frac{n(n-3)}{2}\leq \frac{2n(m-1)}{m-2}<4n.$$ It follows that
$n<11$,  a contradiction. Thus $k\theta_1(1)<d_2$ and so
$k\theta_1(1)=d_1$ which yields that $n-1=k(m-1)$. Since $k\geq
2$,  we deduce that $n\geq 2(m-1)+1=2m-1$. As
$1<k\theta_1(1)<k\theta_2(1)<k\theta_3(1)$,  we see that
$k\theta_3(1)\geq d_3$. Hence $$\frac{k(m-1)(m-2)}{2}\geq
\frac{(n-1)(n-2)}{2}.$$ Substituting $n-1=k(m-1)$ and simplifying,
we have ${m-2}\geq {n-2}$,  which implies that
$m\geq n$. Combining with the previous claim that $n\geq 2m-1$,  we
get a contradiction.

\smallskip
{\bf (1c)} $S$ is a simple group of Lie type in characteristic
$p$ and $S\neq {}^2\textrm{F}_4(2)'$. It is well known that $S$ possesses an
irreducible character $\theta\in\Irr(S)$ of degree $|S|_p\geq 4$
such that $\theta$ extends to $\Aut(S)$. Let $\varphi=\theta^k$. We
know that $\varphi$ extends to $G$ and thus
$\varphi(1)=\theta(1)^k=|S|_p^k\in\cd(G)$. Hence $G$ possesses a
nontrivial prime power degree. By \cite[Theorem~5.1]{BBOO},  $n-1=\theta(1)^k$ and thus $\theta(1)^k=d_1$. Since
$1<k\theta(1)\in\cd(G)$,  we must have $k\theta_1(1)\geq
d_1=\theta_1(1)^k$. However this inequality cannot happen as
$k\geq 2$ and $\theta(1)\geq 4$.

\medskip
{\bf Claim 2.} $S$ is not a sporadic simple group nor the Tits
group. Assume by contradiction that $S$ is a sporadic simple
group or the Tits group. We see that $G/C$ is an almost simple
group with socle $LC/C\cong S$. If $\Out(S)$ is trivial, then $G/C\cong S$ and
hence $\cd(S)\subseteq \cd( \Alt_n)$ so that by
\cite[Theorem~12]{Tong1}, we have $S\cong  \Alt_n$, a
contradiction. Thus we can assume that $\Out(S)$ is nontrivial and
then by \cite{atlas}, $G/C\cong \Aut(S)\cong S\cdot
2$ and so $\cd(S\cdot 2)\subseteq \cd( \Alt_n)$. It follows that
$d_j(S\cdot 2)\geq d_j$ for all $j\geq 1$. As $n\geq 14$,  we
have $$d_2(S\cdot 2)\geq d_2( \Alt_n)\geq 14(14-3)/2=77,$$  hence by checking Table~\ref{Tab2}, $S$ is one of the following simple groups
$$\textrm{HS},\textrm{J}_3,\textrm{McL},\textrm{He},\textrm{Suz},\textrm{O'N},\textrm{Fi}_{22},\textrm{HN},\textrm{Fi}_{24}'.$$ If $d_1(S\cdot 2)=d_1$,  then $n=d_1(S\cdot 2)+1\geq 22$.
But then $d_4(S\cdot 2)<d_4$,  which is impossible. Thus we
can assume that $$d_1(S\cdot 2)\geq d_2=\frac{n(n-3)}{2}.$$ Clearly, $\pi(S\cdot 2)\subseteq
\pi( \Alt_n)$,  so $n\geq n_0$,  where
$n_0=\text{max}\{14,p(S\cdot 2)\}$.

\smallskip
{\bf (2a)} $S\in \{\textrm{J}_3, \textrm{McL}, \textrm{He},  \textrm{HN}\}$. For these groups, we have that $$d_4\geq \frac{n_0(n_0-1)(n_0-5)}{6}> d_2(S\cdot 2)$$ and
thus $$d_2(S\cdot 2)\leq d_3=\frac{(n-1)(n-2)}{2}.$$ As $d_1(S\cdot 2)\geq
d_2$,  we must have that $d_2(S\cdot 2)=d_3$ and $d_1(S\cdot
2)=d_2$ yielding that $d_2(S\cdot 2)=d_1(S\cdot 2)+1$,  which
is impossible by checking Table~\ref{Tab2}.

\smallskip
{\bf (2b)}  $S\in \{\textrm{Suz}, \textrm{Fi}_{22}\}$. If $n\geq 16$,  then $d_4>d_2(S\cdot 2)$ and we can argue as in the previous case to obtain a contradiction. For $14\leq n\leq 15$,  direct calculation using \cite{GAP} shows that $\cd(S\cdot 2)\nsubseteq \cd(\Alt_n)$.

\smallskip
{\bf (2c)}   $S\cong \textrm{O'N}$. As $n\geq 31$,  we have that $d_7\geq
26970>d_2(S\cdot 2)$ and hence $$d_2(S\cdot 2)=26752\leq
d_6=\frac{n(n-2)(n-4)}{3}.$$ Solving this inequality, we have
$n\geq 46$. But then $d_7(S\cdot 2)=58653<d_7$,  a
contradiction.

\smallskip
{\bf (2d)}   $S\cong \textrm{Fi}_{24}'$. As $n\geq 29$,  we have that
$d_7\geq 20097>d_1(S\cdot 2)$ and hence $$d_1(S\cdot 2)=8671\leq
d_6=\frac{n(n-2)(n-4)}{3}.$$ Solving this inequality, we obtain that
$n\geq 32$ and then $d_5\geq 8990>d_1(S\cdot 2)$ so
$$d_1(S\cdot 2)\in \{d_2,d_3,d_4\}.$$ As $d_2\leq
d_1(S\cdot 2)$,  we deduce that $n\leq 133$. However the equations $d_1(S\cdot 2)=d_j$ for $j=2,3,4$,  have no integer solution $n$ in the range $32\leq n\leq 133$.

\medskip
{\bf Claim 3.} $S$ is not a simple groups of Lie type. Assume that $LC/C\cong S$,  where $S\neq
{}^2\textrm{F}_4(2)'$ is a simple group of Lie type in characteristic $p$.  Let $\theta\in\Irr(S)$ be the
Steinberg character of $S$. Arguing as in (1c) above, we have that $n-1=\theta(1)=|S|_p$. Thus
$|S|_p=d_1$ is the smallest nontrivial degree of~$G$. By Lemma \ref{lem3}, we must have
$S\cong \PSL_2(q)$,  where $q=p^f$ for some integer  $f\geq 1$. Hence $G/C$ is an almost simple group
with socle $LC/C\cong \PSL_2(q)$. Observe that $q+1$ is the largest character degree of $\PSL_2(q)$. Now
let $\mu\in\Irr(S)$ be any nontrivial irreducible character of $S$. As $|\Out(S)|=(2,q-1)f\leq
q=p^f$,  if $\chi\in\Irr(G/C)$ is an irreducible constituent of $\mu^{G/C}$,  then $\chi(1)\leq
|G:LC|\mu(1)\leq |\Out(S)|(q+1)\leq q(q+1)$. Hence $q(q+1)$ is an upper bound for all  degrees of $G/C$. As $n-1=|S|_p$,  we have that $n=q+1$. Since $n\geq 14$
$$q(q+1)=(n-1)n<\frac{n(n-1)(n-5)}{6}=d_4$$  which yields that for all $\chi\in\Irr(G/C)$,
$\chi(1)\leq d_3$ and so $|\cd(G/C)|\leq 4$. By \cite[Theorem~12.15]{Isaacs}, we deduce that
$|\cd(G/C)|= 4$ as $G/C$ is nonsolvable. By  \cite[Corollary ~B]{Malle05},
$\cd(G/C)$ is either $\{1,s-1,s,s+1\}$ or $\{1,9,10,16\}$,  where $s$ is some prime power.

Assume first that $\cd(G/C)=\{1,s-1,s,s+1\}$. Then $s-1\ge d_1=n-1=q$.
Since
$s+1 \leq d_3=(n-1)(n-2)/2$,  we deduce that $s+1=(n-1)(n-2)/2$,
$s=n(n-3)/2$ and $s-1=n-1$.
The latter equation implies that $s=n$.
But then as $s=n(n-3)/2$,
we have $n=5<14$,  a contradiction.

Assume $\cd(G/C)=\{1,9,10,16\}$. Then $d_1(G/C)=9\geq d_1=n-1$,  which implies that $n-1\le 9$, so $n\le 10<14$,  a contradiction.

\medskip
{\bf Claim 4.} Show $S\cong  \Alt_n$. We have shown that $S\cong  \Alt_m$,
where $m\geq 7$. It suffices to show that $m=n$. As in the proof of
(1b) above, $m-1,m(m-3)/2$ and $(m-1)(m-2)/2$ are
degrees of $G$. We see that $m-1\geq d_1=n-1$ so
$m\geq n\geq 14$. If $m-1\geq d_2$,  then $$m\geq
d_2+1=d_3=(n-1)(n-2)/2$$ and so as $n\geq 14$,  we have
$d_3\geq 2n$,  and thus $m\geq 2n>n$. By Lemma \ref{lem2}, there
is a prime $p$ such that $n<p\leq 2n\leq m$. It follows that
$p\in\pi(S)\setminus\pi( \Alt_n)$,  which is a contradiction. This shows that
$m-1<d_2$ and hence $m-1=d_1=n-1$. Thus $m=n$ as required. \end{proof}

\section{Finite perfect groups}\label{sec:Largedegrees}

In this section, we prove some properties on the character degree set of a finite perfect group having a special normal structure. Recall that $k(G)$ and $b(G)$ denote the number of conjugacy classes and the largest degree, respectively, of a finite group $G$. We begin with the following result.
\begin{lemma}\label{lem:large degree}
Let $G$ be a finite perfect group and let $M$ be a minimal normal elementary abelian subgroup of $G$. Suppose that $\Centralizer_G(M)=M$ and $G/M\cong \Alt_n$ with $n\geq 17$. Then $b(G)>b(\Alt_n)$. \end{lemma}
\begin{proof}
Suppose by contradiction that $b(G)\le b(G/M)$. Form the semidirect product $TM$ with $T\cong G/M\cong\Alt_n$ acting on $M$ as inside $G$. It follows from \cite[Proposition~2.4]{Guralnick} that $k(G)\leq k(TM)$. Since $n\geq 17$,  by \cite[Theorem~1.4]{Guralnick} we have that $k(TM)\le |M|/2$ and hence $k(G)\le k(TM)\le |M|/2$. We have that \[|G|=\sum_{\chi\in\Irr(G)}\chi(1)^2\le k(G)b(G)^2\le \frac{1}{2}|M|b(G/M)^2.\]
Since $|G|=|G/M|\cdot |M|$,  we deduce that \[|G/M|\cdot |M|\le \frac{1}{2}|M|b(G/M)^2.\] After simplifying, we obtain that \[|G/M|\le \frac{1}{2}b(G/M)^2<b(G/M)^2,\] which is impossible. Therefore, $b(G)>b(G/M)$ as wanted.
\end{proof}

We will need the following result whose proof is similar to the proof of Step 3 in \cite[Section~5]{NTW}. So we only give a sketch.

\begin{lemma}\label{small degrees} Let $n\in \{14,15,16\}$. Let $G$ be a finite perfect group and $M$ be a minimal normal elementary abelian subgroup of $G$ such that $G/M\cong\Alt_n$,  and $\Centralizer_G(M)=M$. Then some degree of $G$ divides no degree of $\Alt_n$. \end{lemma}

\begin{proof}
Suppose by contradiction that every degree of $G$ divides some degree of $\Alt_n$. Let $1_M\neq \theta\in\Irr(M)$. We claim that $\theta$ is $G$-invariant. By \cite[Lemma~6]{Hupp}, $[M,G]=1$,  which implies that $\Centralizer_G(M)=G$,  a contradiction.

By way of contradiction, assume that $1_M\neq \theta\in\Irr(M)$ is not $G$-invariant and let $I=I_G(\theta)$. Let $U$ be a subgroup of $G$ such that $I/M\le U/M$ and $U/M$ is a maximal subgroup of $G/M\cong\Alt_n$. Let $t=|U:I|=|U/M:I/M|$ and write \[\theta^I=\sum_{i=1}^\ell e_i\phi_i,\text{~where $\phi_i\in\Irr(I|\theta)$}.\]

By Clifford Correspondence, for each $i$,  we have $$\phi_i^G(1)=|G:U|\cdot |U:I|\phi_i(1)=t|G:U|\phi_i(1)\in \cd(G)$$ and thus it  divides some degree of $\Alt_n$.  Let $\mathcal{A}$ be the set consisting of all the numbers $\chi(1)/|G:U|$ where $\chi\in\Irr(\Alt_n)$ with $|G:U|\mid \chi(1)$. Then $t\phi_i(1)$ divides some number in $\mathcal{A}$ for each $i$. Furthermore, as the index $|G:U|$ divides some degree of $\Alt_n$,  the possibilities for $U/M$ are given in Tables \ref{A14} - \ref{A16}. From these lists, $U/M$ is isomorphic to $(\Sym_5\wr\Sym_3)\cap \Alt_{15}$,
$U/M\cong (\Alt_{m}\times \Alt_k):2,k+m=n,m> k$ or $(\Alt_m\times \Alt_m):2^2$ with $n=2m$.

\medskip
{\bf(1)} $U/M\cong (\Alt_{m}\times \Alt_k):2,k+m=n,m> k$. Let $L/M\cong\Alt_{m}$.  Then $L\unlhd U$ so $L\cap I\unlhd I$.

\smallskip
{\bf (1a)} Assume that $L\le I$. Since $L\unlhd U$,  we have $L\unlhd I$. Hence if $\lambda\in\Irr(L|\theta)$,  then $\lambda(1)$ divides some $\phi_j(1)$ for some $j$,  so $\lambda(1)$ divides some number in $\mathcal{A}$. If $\theta$ extends to $\theta_0\in L$,  then $\theta_0\mu\in\Irr(L|\theta)$ for all $\mu\in\Irr(L/M)$. If $\theta$ does not extend to $I$,  then the set of ramification indices $\{f_j\}_{j=1}^\ell$,  where $\theta^L=f_1\mu_1+\cdots+f_s\mu_s$,  where $\mu_i\in\Irr(L|\theta)$ coincides with the set of the degrees of all faithful irreducible characters of the Schur cover $2\cdot A_m$ by the theory of character triple isomorphisms. In both cases, if one can find $\lambda\in\Irr(L|\theta)$ with $\lambda(1)$ divides no element in $\mathcal{A}$,  then we are done.

\smallskip
{\bf (1b)} Assume that $L\nleq I$. Then $I\leq IL\leq U$ so $|L:I\cap L|=|IL:I|$ divides $t=|U:I|$. Since $L\nleq I$,  $L\cap I\lneq L$ so $L\cap I\leq R\leq L$,  where $R/M$ is maximal in $L/M$ whose index $|L:R|$ divides some number in $\mathcal{A}$. As maximal subgroups of $L/M\cong \Alt_m$ are known, we can obtain a list of such maximal subgroups $R/M$ of $L/M$. Then $|R:I\cap L|\phi_i(1)$ divides some number in $\mathcal{B}$ with $$\mathcal{B}=\{\frac{a}{|L:R|}:|L:R|\mid a\in\mathcal{A}\}.$$  Let $R_1\unlhd R$ such that $R_1/M$ is a nonabelian simple group (if exists). We now repeat the process as above again.

Assume that $R_1\le I\cap L$. Applying the same argument as in case $(1a)$,  we will eventually obtain a contradiction.

Assume that $R_1\nleq I\cap L$ and let $I\cap L\le T\le R_1$ be such that $T/M$ is maximal in $R_1/M$. Then $|T:R_1\cap I|\phi_i(1)$ divides one of the number in $\mathcal{C}$ with $$\mathcal{C}=\{\frac{b}{|R_1:T|}:|R_1:T|\mid b\in\mathcal{B}\}.$$
Repeat the process until we obtain a contradiction by using Lemma \ref{lemmaMoreto}.

\medskip
{\bf(2)} $U/M\cong (\Alt_{m}\times \Alt_m):2^2,n=2m$.
Let $L/M=L_1/M\times L_2/M$,  where $L_i\unlhd L$ and $L_i/M\cong\Alt_m$.  Then $L\unlhd U$ so $L\cap I\unlhd I$. The maximal subgroups of $L/M$ are known. In fact, every maximal subgroup of $L/M$ is either the diagonal subgroup generated by $(a,a)$ with $a\in\Alt_m$ or has the form $L_1/M\times K_2$ or $K_1\times L_2/M$,  where $K_i$ is maximal in $L_i/M$.
If $L\nleq I$,  then we can argue as in (1b) above. If $L\leq I$,  then $M\unlhd L_1\unlhd L\unlhd I$. If $\lambda\in\Irr(L_1|\theta)$,  then $\lambda(1)$ must divide $\phi_j(1)$ for some $j$, by the transitivity of character induction. From this, one can get a contradiction by finding $\lambda \in \Irr(L_1|\theta)$ of large degree.

\medskip
{\bf(3)} $U/M\cong (\Sym_5\wr\Sym_3)\cap \Alt_{15}$. This case only occurs 
when $n=15$. For this case, we have that $|U:I|\phi_i(1)=1$ for all $i$,  which implies that $I/M=U/M$ is nonsolvable and $\phi_i(1)=1$ for all $i$. The latter implies that $I/M$ is abelian, which is impossible.

\medskip
 We demonstrate this strategy by giving a detailed proof for the case $n=14$. The remaining cases can be dealt with  similarly.

\begin{table}[ht]\caption{Maximal subgroups of small index of $\Alt_{14}$ }\label{A14}
\begin{tabular}{l|l}\hline

Subgroup Structure&Index\\\hline
$\Alt_{13}$ & $14$\\
 $\Sym_{12}$&$91$\\
 $(\Alt_{11}\times \Z_3):2$&$364$\\
  $(\Alt_{10}\times \Alt_4):2$&$1001$\\
   $(\Alt_9\times \Alt_5):2$&$2002$\\
   $(\Alt_8\times \Alt_6):2$&$3003$\\
   $(\Alt_7\times \Alt_7):2^2$&$1716$\\\hline

\end{tabular}
\end{table}

\begin{table}[ht]\caption{Maximal subgroups of small index of $\Alt_{15}$ }\label{A15}
\begin{tabular}{l|l}\hline
Subgroup Structure&Index\\\hline
$\Alt_{14}$ & $15$\\
 $\Sym_{13}$&$105$\\
$(\Alt_{12}\times \Z_3):2$&$455$\\
 $(\Alt_{11}\times \Alt_4):2$&$1365$\\
  $(\Alt_{10}\times \Alt_5):2$&$3003$\\
   $(\Alt_9\times \Alt_6):2$&$5005$\\
   $(\Alt_8\times \Alt_7):2$&$6435$\\
   $(\Sym_5\wr\Sym_3)\cap \Alt_{15}$&$126126$\\\hline

\end{tabular}
\end{table}

\begin{table}[ht]\caption{Maximal subgroups of small index of $\Alt_{16}$ }\label{A16}
\begin{tabular}{l|l}\hline

Subgroup Structure&Index\\\hline
$\Alt_{15}$ & $16$\\
 $\Sym_{14}$&$120$\\
$(\Alt_{13}\times \Z_3):2$&$560$\\
 $(\Alt_{12}\times \Alt_4):2$&$1820$\\
  $(\Alt_{11}\times \Alt_5):2$&$4368$\\
   $(\Alt_{10}\times \Alt_6):2$&$8008$\\
   $(\Alt_9\times \Alt_7):2$&$11440$\\
   $(\Alt_8\times \Alt_8):2^2$&$6435$\\\hline

\end{tabular}
\end{table}

From Table \ref{A14}, we consider the following cases:

\medskip
{\bf Case 1:} $U/M\cong \Alt_{13}$. We have that $\mathcal{A}$ consists of the following numbers:

\begin{center}\begin{tabular}{c}40, 143, 312, 352, 429, 546, 858, 975,1001,\\ 1144,1456, 1664, 2002, 3003,
  3432, 3575, 4576
\end{tabular}
\end{center}
Assume first that $t=1$. Then $I/M\cong \Alt_{13}$. If $\theta$ extends to $\theta_0\in\Irr(I)$,  then by Gallagher's Theorem, $\theta_0\tau\in\Irr(I|\theta)$ for all $\tau\in\Irr(I/M)$. Choose $\tau\in\Irr(\Alt_{13})$ with $\tau(1)=21450$,  we obtain a contradiction as $\theta_0(1)\tau(1)$ divides no number in $\mathcal{A}$. Similarly, if $\theta$ is not extendible to $I$,  then one can find $\gamma\in\Irr(I|\theta)$ with $\gamma(1)=20800$ and $\gamma(1)$ divides no elements in $\mathcal{A}$. Notice that in the latter case $20800$ is the degree of a faithful irreducible character of $2\cdot \Alt_{13}$.
Assume that $t>1$. Then $I\leq R\le U$ and $R/M$ is maximal in $U/M$. Since $|U:R|$ divides some number in $\mathcal{A}$,  the possibilities for $R/M$ are given in Table \ref{A14A13}.

\begin{table}[ht]\caption{Maximal subgroups of small index of $\Alt_{13}$ }\label{A14A13}
\begin{tabular}{l|l}\hline
Subgroup Structure&Index\\\hline
$\Alt_{12}$ & $13$\\
 $\Sym_{11}$&$78$\\
  $(\Alt_{10}\times \Z_3):2$&$286$\\
   $(\Alt_{9}\times \Alt_4):2$&$715$\\
   $(\Alt_7\times \Alt_6):2$&$1716$\\\hline
   \end{tabular}
\end{table}

{\bf (1a)}  $R/M\cong \Alt_{12}$. As $|U:R|=13\mid t$,  $|R:I|\phi_i(1)$ divides one of the numbers in $\mathcal{B}$ with
\[\mathcal{B}=\{11, 24, 33, 42, 66, 75, 77, 88, 112, 128, 154, 231, 264, 275, 352\}.\]

(i) Assume $I=R$. Whether $\theta$ extends to $R$ or not, we can find $\lambda\in\Irr(I|\theta)$ such that $\lambda(1)$ does not divide any number above, a contradiction. Indeed, one can choose $\lambda(1)=5775$ if $\theta$ is extendible to $I$ and $\lambda(1)=7776$ if $\theta$ is not extendible.

(ii) Assume that $I\lneq R$. Then $I\leq J\leq R$ where $J/M$ is maximal in $R/M$. As the maximal index $|R:J|$ divides one of the number in $\mathcal{B}$,  $J/M\cong \Sym_{10}$ or $\Alt_{11}$. If the first case holds, then $|J:I|\phi_i(1)$ divides $4$ and if the latter case holds, then $|J:I|\phi_i(1)$ divides $22$. Assume that the former case holds. Investigating the maximal subgroups of $\Sym_{10}$,  as $|J:I|\mid 4$,  we deduce that $|J:I|=1$ or $2$ so $I/M\cong \Sym_{10}$ or $\Alt_{10}$,  in particular, $I/M$ is nonsolvable. However, as $\phi_i(1)\mid 4$ for all $i$,  each $\phi_i(1)$ is a power of $2$. By Lemma \ref{lemmaMoreto}, $I/M$ is solvable, which is a contradiction. So $J/M\cong \Alt_{11}$.  Again, as $|J:I|\mid 22$,  $|J:I|=1$ or $11$. In both cases, $I/M$ is nonsolvable. Now if $I\neq J$,  then $|J:I|=11$ and $\phi_i(1)\mid 2$ for all $i$ and $I/M\cong \Alt_{10}$,  we obtain a contradiction as above. So, $I/M\cong \Alt_{11}$. This case also leads to a contradiction as we can always find $\lambda\in\Irr(I|\theta)$ with $\lambda(1)>22$.

\smallskip

{\bf (1b)}  $R/M\cong \Sym_{11}$. Since $|U:R|=78$,  for each $i$,  $|R:I|\phi_i(1)$ divides $7$ or $44$. Let $M\unlhd R_1\unlhd R$ be such that $R_1/M\cong \Alt_{11}$.

Assume that $R_1\leq I$. Then $R_1\unlhd I$ and so for each $\lambda\in\Irr(R_1|\theta)$,  $\lambda(1)$ divides some $\phi_j(1)$ and so divides $7$ or $44$. Since $R_1/M\cong \Alt_{11}$,  one can choose $\lambda\in\Irr(R_1|\theta)$ with $\lambda(1)>44$. So, assume that $R_1\nleq I$.
Since $R_1\unlhd R$,  we have $I\lneq IR_1\le R$. Thus $|R:I|$ is divisible by $|IR_1:I|=|R_1:I\cap R_1|$. As $I\cap R_1\lneq R_1$ and $R_1/M\cong \Alt_{11}$,  $|R_1:I\cap R_1|$ and so $|R:I|$ is divisible by the index of some maximal subgroup of $\Alt_{11}$. So some maximal index of $\Alt_{11}$ divides $7$ or $44$,  which implies that $11\mid |R:I|$ and hence $\phi_i(1)\mid 4$ for all $i$. Let $M\unlhd K\leq R_1$ such that $K/M\cong \Alt_{10}$. Then $I\cap R_1\leq K$ and $|K:I\cap R_1|\mid 4$.  As the smallest index of $\Alt_{10}$ is $10$,  $K=I\cap R_1$,  so $I/M$ is nonsolvable. But then this contradicts Lemma \ref{lemmaMoreto} as all $\phi_i(1)'s$ are $2$-powers.

\smallskip
{\bf(1c)}  $R/M\cong (\Alt_{10}\times \Z_3):2$. As $|U:R|=286$,  $|R:I|\phi_i(1)$ divides $7,12$ or $16$. Let $M\unlhd R_1\unlhd R$ be such that $R_1/M\cong \Alt_{10}$. If $R_1\leq I$,  then by considering the character degree sets of $\Alt_{10}$ and $2\cdot \Alt_{10}$,  we obtain a contradiction as in the previous case.
So, assume that $R_1\nleq I$. As in the proof of the previous case, $|IR_1:I|=|R_1:R_1\cap I|$ is divisible by the index of some maximal subgroup of $\Alt_{10}$,  and so $\Alt_{10}$ has some maximal subgroup whose index divides $7,12$ or $16$,  which is impossible.

\smallskip
{\bf (1d)}  $R/M\cong (\Alt_{9}\times \Alt_4):2$ or $(\Alt_{7}\times \Alt_6):2$ As $|U:R|=715$ or $1716$,  $|R:I|\phi_i(1)$ divides $5$ or $2$,  respectively. It follows that for each $i$,  $\phi_i(1)$ is a power of a fixed prime and $I/M$ is nonsolvable, contradicting Lemma \ref{lemmaMoreto}.

\medskip
{\bf Case 2:} $U/M\cong \Sym_{12}$. Let $M\unlhd L\unlhd U$ be such that $L/M\cong \Alt_{12}$. In this case, the largest element in $\mathcal{A}$ is $704$.
If $L\leq I$,  then one can find $\lambda\in\Irr(L|\theta)$ with $\lambda(1)=5775$ or $7776$ according to whether $\theta$ extends to $L$ or not and we get a contradiction as $\lambda(1)>704$,  the largest number in $\mathcal{A}$. So, we assume that $L\nleq I$.  Then $|U:I|$ is divisible by $|IL:I|=|L:L\cap I|$ with $L\cap I\lneq L$. Let $M\unlhd R\leq L$ be such that $R/M$ is maximal in $L/M$ and $I\cap L\le R$. It follows that the maximal index $|L:R|$ divides some number in $\mathcal{A}$. Then one of the following cases holds.

\smallskip
(i) $R/M\cong \Alt_{11}$. Then $|L:R|=12$ and $|R:I\cap L|\phi_i(1)$ divides $7$ or $44$.
Assume that $R=I\cap L$. Since $L\unlhd U$,  $R=I\cap L\unlhd I$. So, for every $\lambda\in\Irr(R|\theta)$,  $\lambda(1)$ divides some $\phi_j(1)$ for some $j$,  and hence divides $7$ or $44$. However, this is impossible as $R/M\cong \Alt_{11}$.

Assume that $I\cap L\lneq R$. Let $M\unlhd T\leq R$ be such that $I\cap L\le T$ and $T/M$ is maximal in $R/M$. Then $|R:T|$ divides $7$ or $44$ which implies that $T/M\cong \Alt_{10}$ and $|R:T|=11$. Hence $|T:I\cap L|\phi_i(1)\mid 4$ for all $i$. This implies that all $\phi_i(1)$ are powers of $2$ and $I/M$ is nonsolvable, a contradiction.

\smallskip
(ii) $R/M\cong \Sym_{10}$. Then $|L:R|=66$ and $|R:I\cap L|\phi_i(1)$ divides $7$ or $8$. Let $R_1\unlhd R$ such that $R_1/M\cong \Alt_{10}$. We can check that $R_1\leq I$ as the smallest index of $\Alt_{10}$ is $10$ which is larger than $8$. But then one can find $\lambda\in\Irr(R_1|\theta)$ such that $\lambda(1)>8$.

\smallskip
(iii) $R/M\cong (\Alt_6\times \Alt_6):2^2$. Then $|L:R|=462$ and $|R:I\cap L|\phi_i(1)=1$. This case obviously cannot happen as $I/M$ contains $R/M$ which is nonsolvable and all $\phi_i(1)=1$.

\begin{table}[ht]\caption{Maximal subgroups of small index of $\Sym_{12}$}\label{A14S12}
\begin{tabular}{l|l}\hline
Subgroup Structure&Index\\\hline
$\Alt_{12}$ & $2$\\
 $\Sym_{10}\times \Sym_2$&$66$\\
  $\Sym_{11}$&$12$\\
   $\Sym_6\wr \Sym_2$&$462$\\\hline
   \end{tabular}
\end{table}

\medskip
{\bf Case 3:} $U/M\cong (\Alt_{11}\times \Z_3):2$. We have $L/M\cong\Alt_{11}$ and $|U:I|\phi_i(1)$ divides one of the numbers
\[12, 21, 33, 44, 56, 64, 77, 132, 176.\]
Arguing as before, we can assume that $L\nleq I$. So $I\lneq IL\le U$ and $|U:I|$ is divisible by $|L:R|$ with $I\cap L\leq R\leq L$ and $R/M$ maximal in $L/M$. It follows that $R/M\cong \Alt_{10}$ and $|R:I\cap L|\phi_i(1)$ divides $7,12$ or $16$. Observe that $I\cap L\neq R$ as $\Irr(R|\theta)$ possesses irreducible character of degree strictly larger than $16$. Thus $I\cap L\le K\le R$ with $K/M$ maximal in $R/M\cong \Alt_{10}$ and $|R:K|$ divides $7,12$ or $16$. However, this is impossible by investigating the maximal subgroups of $\Alt_{10}$.

\medskip
{\bf Case 4:} $U/M\cong (\Alt_{10}\times \Alt_4):2$. The largest element in $\mathcal{A}$ is $64$. Let $L/M\cong \Alt_{10}$. As above, we deduce that $L\nleq I$ and if $I\cap L\le R\le L$ with $R/M$ maximal in $L/M$,  then $R/M\cong\Alt_9$ and $|R:I\cap L|\phi_i(1)\mid 5$. Clearly, this case cannot happen.

\medskip
{\bf Case 5:} $U/M\cong (\Alt_{9}\times \Alt_5):2$. Then $$\mathcal{A}=\{ 1, 3, 6, 7, 8, 14, 21, 24, 25, 32 \}.$$  Let $L/M\cong \Alt_{9}$. As above, we have $L\nleq I$ but no maximal index of $L/M$ divides a number in $\mathcal{A}$.

\medskip
{\bf Case 6:} $U/M\cong (\Alt_{8}\times \Alt_6):2$. Let $L/M\cong\Alt_8$. We have \[\mathcal{A}=\{2, 3, 4, 5, 7, 9, 14, 16, 21\}.\]
So, $L\nleq I$ and if $I\cap L\leq R\leq L$ with $R/M$ is maximal in $L/M$,  then $R/M\cong \Alt_7$ and $|R:I\cap L|\phi_i(1)\mid 2$ for all $i$.

\medskip
{\bf Case 7:} $U/M\cong (\Alt_{7}\times \Alt_7):2^2$. Let $L/M\cong \Alt_7\times \Alt_7$. We have \[\mathcal{A}=\{ 1, 7, 9, 20, 28\}.\] Every maximal subgroup of $L/M=L_1/M\times L_2/M,L_i/M\cong\Alt_7$,  is either the diagonal subgroup generated by $(a,a)$ with $a\in\Alt_7$ or has the form $L_1/M\times K_2$ or $K_1\times L_2/M$,  where $K_i$ is maximal in $L_i/M$. If $L\nleq I$,  then $I\cap L\le R\le L$ with $R/M\cong A_6\times A_7$ or $A_7\times A_6$ and $|R:I\cap L|\phi_i(1)\mid 4$. Clearly, this case cannot happen. Thus $L\unlhd  I\unlhd U$ with $|U:L|=4$. Since $L_1$ is subnormal in $I$ and there exists $\lambda\in\Irr(L_1|\theta)$ with $\lambda(1)>28$,  so one can find $j$ such that $\phi_j(1)>28$,  a contradiction.
\end{proof}

\begin{theorem}\label{th:largedegree}
Let $n\in\N,n\geq 14$. Let $G$ be a finite perfect group and $M$ be a minimal normal elementary abelian subgroup of $G$ such that $G/M\cong\Alt_n$,  and $\Centralizer_G(M)=M$. Then some degree of~$G$ divides no degree of~$\Alt_n$. \end{theorem}

\begin{proof}
Clearly, if $14\le n\le 16$,  then the theorem follows from Lemma \ref{small degrees}. Now assume that $n\geq 17$. By Lemma \ref{lem:large degree}, the largest degree of~$G$ is strictly larger than~$b(\Alt_n)$,  so this degree divides no degree of~$\Alt_n$ as wanted.
\end{proof}
\section{Proof of the main theorems}\label{sec:HCproof}
We now prove our main results. In the first theorem, we obtain the structure of the finite groups $G$ under the assumption that $\cd(G)=\cd(\Alt_n)$ with $n\ge 14$ using the results we have proven so far.  Our main theorem will follow by combining this with the result due to Debaene \cite{Debaene}.

\begin{theorem}\label{thm:almostHC}
Let $n\in \N$, $n\geq 14$.
Let $G$ be a finite group such that
$\cd(G)=\cd( \Alt_n)$.
Then $G$ has a normal abelian subgroup $A$ such that one of the following holds:
\begin{enumerate}
\item[$(i)$] $G\cong \Alt_n\times A$  (so Huppert's Conjecture is confirmed).
\item[$(ii)$] $G\cong (\Alt_n\times A)\cdot 2$ and $G/A\cong \Sym_n$.
\end{enumerate}
\end{theorem}
\begin{proof}
Let $R$ be the solvable radical of $G$ and let $L$ be the last term of the derived series of~$G$. By Theorem~\ref{thm:nonsolvable}, $G$ is nonsolvable and thus $L$ is a nontrivial normal perfect subgroup of~$G$.
Let $D/R$ be a chief factor of~$G$. Clearly, $D/R$ is nonabelian and thus $D/R\cong\Alt_n$ by Theorem~\ref{th:composition}. Now let $C$ be a normal subgroup of $G$ such that $C/R=\Centralizer_{G/R}(D/R)$.
Then $G/C$ is almost simple with simple socle $DC/C\cong \Alt_n$.
Since $n\geq 14$, $\Aut(\Alt_n)\cong\Sym_n$ and thus  $G/C\cong \Alt_n$ or  $\Sym_n$.
We see that $$DC/R=D/R\times C/R\cong \Alt_n\times C/R.$$

We claim that  $C/R$ is abelian. If this is not the case,
let $\chi \in \Irr(D/R) $ with $\chi(1)=b(\Alt_n)$ and
$\lambda \in \Irr(C/R)$ with $\lambda(1) > 1$ then $\chi \lambda \in \Irr(DC/R)$
with degree $\chi(1)\lambda(1)=b(\Alt_n)\lambda(1)>b(\Alt_n)$. Since $DC\unlhd G$,  $\chi(1)\lambda(1)$ divides some degree of~$G$,  which is impossible.
Thus $C/R$ is abelian as claimed and so $C=R$.
Since $LR/R$ is a perfect normal subgroup of the almost simple group $G/R$ with simple socle $D/R\cong\Alt_n$,  we deduce that $LR/R=D/R\cong \Alt_n$,  and $G/R\cong\Alt_n$ or $\Sym_n$,  hence $|G:LR|\le 2$.

Let $V:=R\cap L$. Then $V\unlhd G$ and $LR/R\cong L/V\cong\Alt_n$,  so $LR/V\cong L/V\times R/V\cong \Alt_n\times R/V$. Since $LR\unlhd G$,  argue as above, we deduce that $R/V$ is abelian.
If $V$ is trivial, then $\Alt_n\times R=L\times R\cong LR\unlhd G$,  where $R$ is abelian.
Now if $G/R\cong \Alt_n$,  then $G=L\times R$ and conclusion $(i)$ holds. If $G/R\cong\Sym_n$,  then $|G:LR|=2$ so $G=(L\times R)\cdot 2$,  hence conclusion $(ii)$ holds.
So, assume that $V$ is nontrivial. Let $V/U$ be a chief factor of $L$    and let $\overline{L}=L/U$. Since $V\leq R$,  $V$ is solvable and thus $\overline{V}$ is a normal elementary abelian subgroup of the perfect group $\overline{L}$.

\smallskip
{\bf(a)}  $V/U=\Center(L/U)\cong\Z_2$ and $L/U\cong2\cdot \Alt_n$. Let $\overline{W}=\Centralizer_{\overline{L}}(\overline{V})$. Then $\overline{V}\unlhd \overline{W}\unlhd \overline{L}$. As $\overline{L}/\overline{V}\cong L/V\cong \Alt_n$,  either $\overline{W}=\overline{L}$ or $\overline{W}=\overline{V}$. Assume that the latter case holds.
As $\cd(\overline{L})\subseteq \cd(L)$ and $L\unlhd G$,  every degree of~$\overline{L}$ divides some degree of~$G$,   contradicting Theorem \ref{th:largedegree}.
Thus $\overline{W}=\overline{L}$ so $\overline{V}\leq \Center(\overline{L})\cap (\overline{L})'$. Hence $|\overline{V}|$ divides the order of the Schur multiplier of $\overline{L}/\overline{V}\cong \Alt_n$ (see the proof of \cite[Lemma~6]{Hupp}). Since $n\geq 14$,  the Schur multiplier of $\Alt_n$ is cyclic of order $2$ and the universal covering group of $\Alt_n$ is the double cover $2\cdot \Alt_n$. As $|\overline{V}|>1$,  we have $\overline{V}=\Center(\overline{L})\cong\Z_2$ and  $\overline{L}\cong 2\cdot \Alt_n$ as wanted.

\smallskip
{\bf(b)} $U\unlhd G$. Suppose that $U$ is not normal in $G$. Clearly, the core of $U$ in $G$ defined by $U_G:=\cap_{g\in G}U^g$ is  the largest normal subgroup of $G$ contained in $U$. Let $K\unlhd G$ be such that $K\lneq U\le V$ and $V/K$ is a chief factor of $G$. (Noting that $K$ could be trivial). For each $g\in G$,  we have $K=K^g\leq U^g$ and so $K\le U_G\le V\unlhd G$. Since $V/K$ is a chief factor of $G$,  we deduce that $K=U_G$. From ${(a)}$,  we know that $V/U=\Center(L/U)$,  so $[L,V]\le U$. Since both $L$ and $V$ are normal in $G$,  $[L,V]\unlhd G$ and thus $[L,V]\le U_G=K$. Now $L/K$ is a perfect group with a central subgroup $V/K$ such that $(L/K)/(V/K)\cong L/V\cong\Alt_n$. It follows that $L/K\cong 2\cdot \Alt_n$ and $V/K\cong \Z_2$. Therefore $K\lneq U\lneq V$ with $|V/U|=2$ and $|V/K|=2$,  which is impossible. Thus $U\unlhd G$ as wanted.

\smallskip
{\bf(c)} The final contradiction.
By $(a)$ and ${(b)}$,  $U\unlhd G$,  $V/U=\Center(L/U)\cong\Z_2$ and $L/U\cong 2\cdot \Alt_n$. Since $V/U\unlhd G/U$ and $|V/U|=2$,  $V/U\le \Center(G/U)$,  so $[G,V]\le U$. Recall that $R/V$ is abelian.
We have $[L,R]=[R,L]\leq L\cap R=V$,  hence \[[L,R,L]\le [V,L]\le U\text{ and }[R,L,L]\le [V,L]\le U.\] By Three Subgroups Lemma, we have $[L,L,R]=[L,R]\le U$.
It follows that $LR/U=L/U\circ R/U$ is a central product with $L/U\cap R/U=V/U\cong \Z_2$. Let $\alpha\in\Irr(V/U)$ be a nontrivial irreducible character.
Since $(R/U)'=R'U/U\subseteq V/U\subseteq \Center(R/U)$,  $R/U$ is nilpotent. Then $R/U=P/U\times Q/U$,  where $P/U$ is a Sylow $2$-subgroup and $Q/U$ is a normal $2$-complement in $R/U$. Obviously $V/U\unlhd P/U$ and $V/U$ is centralized by $Q/U$. We can find $\lambda_0\in\Irr(P/U|\alpha)$ with $\lambda_0(1)=2^a$ for some $a\geq 0$. Clearly, $\lambda=\lambda_0\times 1_{Q/U}\in\Irr(R/U|\alpha)$ with $\lambda(1)=2^a$. As $L/U\cong 2\cdot \Alt_n$,  by \cite[Theorem~4.3]{Bessenrodt}, we can find $\nu\in\Irr(L/U|\alpha)$ with $\nu(1)=2^{\lfloor(n-2)/2\rfloor}$. By \cite[Lemma~5.1]{IMN}, $\phi:=\nu\cdot \lambda\in \Irr(LR/U)$ of degree  $\nu(1)\lambda(1)=2^{a+\lfloor(n-2)/2\rfloor}$. Since $|G:LR|\le 2$,  if $\chi\in\Irr(G|\phi)$,  then $\chi$ is an extension of $\phi$ or $\chi=\phi^G$. Hence either $\phi(1)$ or $2\phi(1)$ is a degree of~$G$.  Now \cite[Theorem~5.1]{BBOO} yields
\begin{equation}\label{eqn5} n-1=2^{\epsilon+a+\lfloor(n-2)/2\rfloor},
\end{equation} where $\epsilon=0$ or $1$.
Since
$$\epsilon+a+\bigg\lfloor\frac{n-2}{2}\bigg\rfloor\geq \frac{n-2}{2}-1=\frac{n-4}{2},$$
we deduce that
$n-1\geq 2^{(n-4)/2}$. As $n\geq 14$,  by using induction on $n$ the latter inequality cannot occur, so \eqref{eqn5} cannot happen.
The proof is now complete.
\end{proof}

Finally, we can give the \textbf{proof of Theorem \ref{th:main}}.
Let $G$ be a finite group such that $\cd(G) = \cd(\Alt_n)$, $n\geq 5$.
We may assume that $n\geq 14$, as the result was already proved up to $n=13$.
If we are in case (i) of Theorem \ref{thm:almostHC}, then Huppert's Conjecture holds and we are done. So assume case (ii) of the theorem occurs. It follows that $\cd(\Sym_n)=\cd(G/A)\subseteq\cd(\Alt_n)$.
Now using the main result in \cite{Debaene} claiming that $\cd(\Sym_n)\not\subseteq\cd(\Alt_n)$ we obtain a contradiction.
Hence Theorem \ref{th:main} now follows.


\end{document}